\documentclass[11pt]{amsart}
\usepackage{graphicx}
\usepackage[pagebackref, colorlinks=true, linkcolor=black, citecolor=black, urlcolor=black]{hyperref}

\usepackage{xargs}                      
\usepackage[colorinlistoftodos,prependcaption 
]{todonotes}                                                                                                                 

\newcommandx{\todoN}[2][1=]{\todo[linecolor=red,backgroundcolor=white,bordercolor=red,#1]{#2}}			

\usepackage{comment}
\usepackage{calc}

\usepackage{amsmath, amsthm, amsfonts, amssymb}
\usepackage{enumerate}
\usepackage{bbm}
\usepackage{mathrsfs}
\usepackage{amssymb}
\usepackage{mathtools}
\usepackage[all]{xy}
\usepackage{tikz}
\usepackage{tikz-cd}
\usetikzlibrary{matrix,arrows,decorations.pathmorphing}
\usepackage{xcolor}
\usepackage{bm}
	\usetikzlibrary{calc}
	\usetikzlibrary{trees}
	\tikzstyle{every picture}=[scale=.35,inner sep=0]
	\usepackage{color}

\usepackage[boxsize=0.7em]{ytableau} 

\newtheorem{thm}{Theorem}

\newtheorem{cor}{Corollary}
\theoremstyle{definition}

\theoremstyle{theorem}
\newtheorem{theorem}{Theorem}[section]
\newtheorem{lemma}[theorem]{Lemma}
\newtheorem{proposition}[theorem]{Proposition}

\theoremstyle{definition}
\newtheorem{definition}[theorem]{Definition}

\theoremstyle{remark}
\newtheorem{remark}[theorem]{Remark}

\numberwithin{equation}{section}

\makeatletter
\@namedef{subjclassname@2020}{\textup{MSC2020}}
\makeatother

\begin{document}

\title{A pointed Prym-Petri Theorem}

\author[N.~Tarasca]{Nicola Tarasca}
\address{Nicola Tarasca 
\newline \indent Department of Mathematics \& Applied Mathematics
\newline \indent Virginia Commonwealth University, Richmond, VA 23284}
\email{tarascan@vcu.edu}


\begin{abstract}
We construct pointed Prym-Brill-Noether varieties pa\-ram\-e\-trizing line bundles assigned to an irreducible \'etale double covering of a curve with a prescribed 
minimal vanishing at a fixed point. We realize them as degeneracy loci in type D and deduce their  classes in case of expected dimension. Thus, we determine a pointed Prym-Petri map and prove a pointed version of the Prym-Petri theorem implying that the expected dimension holds in the general case. These results build on  work of Welters and De Concini--Pragacz on the unpointed case. 
Finally, we show that  Prym varieties are Prym-Tyurin varieties for Prym-Brill-Noether curves of exponent enumerating standard shifted tableaux times a factor of $2$, extending to the Prym setting work of Ortega.
\end{abstract}

\vspace*{-1.5pc}

\maketitle

\vspace{-1.5pc}

\section*{Introduction}


Prym-Brill-Noether varieties parametrize line bundles assigned to an irreducible \'etale double covering of an algebraic curve with a prescribed minimal number of independent global sections.
Introduced by Welters \cite{welters1985theorem}, they have been endowed with the scheme structure  of a degeneracy locus  by \mbox{De Concini--Pragacz} \cite{de1995class}.
Specifically, following a result by Mumford \cite{mumford1971theta}, Prym-Brill-Noether varieties are realized as the  loci where two maximal isotropic subbundles of a vector bundle with a non-degenerate quadratic form intersect in a prescribed minimal dimension. 

This study presents a refinement of  Prym-Brill-Noether varieties obtained by imposing vanishing conditions at a fixed point. The goal is to extend some of the fundamental  results on pointed Brill-Noether varieties from Eisenbud-Harris \cite{MR846932} and Ciliberto-Harris-Teixidor-i-Bigas \cite{ciliberto1992endomorphisms} to the Prym setting.
For this, we construct pointed Prym-Brill-Noether varieties  as degeneracy  loci with scheme structure induced by Schubert  varieties in an orthogonal Grassmannian. Thus, results from Anderson-Fulton \cite{anderson2018chern} in type D apply to compute their expected dimension and class. 

The varieties are described as follows.
Let $C$ be a smooth algebraic curve of genus $g\geq 2$. A nontrivial $2$-torsion point $\epsilon$ in $\mathrm{Jac}(C)$ determines an irreducible \'etale double covering $\varphi\colon \widetilde{C}\rightarrow C$,
and this induces a norm map $\mathrm{Nm}\colon \mathrm{Pic}(\widetilde{C}) \rightarrow \mathrm{Pic}(C)$.
The scheme-theoretic inverse image $\mathrm{Nm}^{-1}(\omega_C)$ has two components:
\begin{align*}
\mathrm{Pr}^+ (C,\epsilon)&:=\left\{ L\in \mathrm{Pic}^{2g-2}(\widetilde{C}) \, : \, \mathrm{Nm}(L)= \omega_C, \, h^0(\widetilde{C},L)\equiv 0 \mbox{ mod } 2   \right\},\\
\mathrm{Pr}^- (C,\epsilon)&:=\left\{ L\in \mathrm{Pic}^{2g-2}(\widetilde{C}) \, : \, \mathrm{Nm}(L)= \omega_C, \, h^0(\widetilde{C},L)\equiv 1 \mbox{ mod } 2   \right\}.
\end{align*}
These are both translates of the Prym variety $\mathrm{Pr} (C,\epsilon)$ given by the connected component of the origin in the kernel of $\mathrm{Nm} \colon \mathrm{Jac}(\widetilde{C}) \rightarrow \mathrm{Jac}(C)$. In particular, the Prym variety is an abelian subvariety of $\mathrm{Jac}(\widetilde{C})$ of dimension $g-1$, and the canonical polarization on $\mathrm{Jac}(\widetilde{C})$ restricts to twice a principal polarization~$\Xi$ on $\mathrm{Pr} (C,\epsilon)$ \cite{mumford1974prym}.

For $\bm{a}=(0\leq a_0 < a_1< \dots< a_r)$ and a point $P$ in $\widetilde{C}$, we consider the   closed subset in $\mathrm{Nm}^{-1}(\omega_C)$ assigned to $(C, \epsilon, P)$  
\[
\mathrm{V}^{\bm{a}}(C, \epsilon, P) := \left\{ L\in \mathrm{Pic}^{2g-2}(\widetilde{C}) \, \left\vert \, 
\begin{array}{l}
\mathrm{Nm}(L)= \omega_C,  \\[0.2cm]
h^0(\widetilde{C},L)\equiv r+1 \mbox{ mod } 2, \\[0.2cm]
h^0(\widetilde{C}, L(-a_i\,P))\geq r+1-i, \,\, \forall \, i 
\end{array}
\right.
\right\}.
\]
The  Prym-Brill-Noether variety from \cite{welters1985theorem} is recovered set-the\-o\-ret\-ically when no special condition is imposed at the point $P$, i.e., for $\bm{a}=(0,1,\dots, r)$.
In \S\ref{sec:degstr}, we endow  $\mathrm{V}^{\bm{a}}(C, \epsilon, P)$ with the scheme structure of a degeneracy locus in type D, generalizing the case $\bm{a}=(0,1,\dots, r)$ from  \cite{de1995class}. We call $\mathrm{V}^{\bm{a}}(C, \epsilon, P)$ the \textit{pointed Prym-Brill-Noether variety} assigned to $(C, \epsilon, P)$.

We deduce that the expected dimension of $\mathrm{V}^{\bm{a}}(C, \epsilon, P)$ is
\begin{equation}
\label{eq:beta}
\beta(g, \bm{a}) := g-1 - |\bm{a}|
\end{equation}
where $|\bm{a}|:=\sum_{i=0}^r a_i$, 
and show that a Pfaffian formula from \cite{anderson2018chern} applies to compute the class of $\mathrm{V}^{\bm{a}}(C, \epsilon, P)$ when the expected dimension holds.

Specifically, let $\ell(\bm{a})$ be the number of positive elements of $\bm{a}$. The theta divisor $\Xi$ on  $\mathrm{Pr} (C,\epsilon)$ translates to a theta divisor on $\mathrm{Pr}^\pm (C,\epsilon)$. Let $\xi$ be the class of such theta divisor in the numerical equivalence ring $\mathrm{N}^*\left(\mathrm{Pr}^\pm (C, \epsilon), k\right)$ or in the singular cohomology ring $\mathrm{H}^*\left(\mathrm{Pr}^\pm (C, \epsilon), \mathbb{C}\right)$.
Define the class
\begin{equation}
\label{eq:classB}
\mathsf{B}(g,\bm{a}):= 2^{|\bm{a}|-\ell(\bm{a})}\, \prod_{i=0}^r \frac{1}{a_i!}\, \prod_{0\leq j<i\leq r}\frac{a_i-a_j}{a_i+a_j} \, \xi^{|\bm{a}|}
\end{equation}
in $\mathrm{N}^*\left(\mathrm{Pr}^\pm (C, \epsilon), k\right)$ for a ground field $k$ of characteristic different from $2$, and in  $\mathrm{H}^*\left(\mathrm{Pr}^\pm (C, \epsilon), \mathbb{C}\right)$. 

\begin{thm}
\label{thm:deglocus}
One has $\dim (\mathrm{V}^{\bm{a}}(C, \epsilon, P))\geq \beta(g,\bm{a})$. If equality holds, then $\mathrm{V}^{\bm{a}}(C, \epsilon, P)$ is 
a reduced Cohen-Macaulay and normal scheme, and 
\begin{equation*}
\left[ \mathrm{V}^{\bm{a}}(C, \epsilon, P) \right] = \mathsf{B}(g,\bm{a})
\end{equation*}
in $\mathrm{N}^*\left(\mathrm{Pr}^\pm (C, \epsilon), k\right)$ when \mbox{$\mathrm{char}(k) \neq 2$} and in $\mathrm{H}^*\left(\mathrm{Pr}^\pm (C, \epsilon), \mathbb{C}\right)$.
\end{thm}

For arbitrary $(C, \epsilon, P)$, the class $\mathsf{B}(g,\bm{a})$ is supported on the pointed Prym-Brill-Noether variety also when $\dim (\mathrm{V}^{\bm{a}}(C, \epsilon, P))> \beta(g,\bm{a})$. 
Since $\xi$ is ample, $\mathsf{B}(g,\bm{a})\neq 0$ when \mbox{$\beta(g,\bm{a}) \geq 0$.}
Hence, we deduce:

\begin{cor}
If \mbox{$\beta(g,\bm{a}) \geq 0$,} then $\mathrm{V}^{\bm{a}}(C, \epsilon, P)$ is non-empty  and of dimension at least $\beta(g,\bm{a})$ for \textit{all} $(C, \epsilon, P)$ as above.
\end{cor}

\noindent The above argument is modeled on the case $\bm{a}=(0,1,\dots, r)$, first treated in \cite{de1995class}.
Our main result here shows that the expected dimension is generically satisfied:

\begin{thm}
\label{thm:mainintro}
For a general curve $C$ of genus $g$, an arbitrary nontrivial {$2$-torsion} point $\epsilon$ in $\mathrm{Jac}(C)$, and a general point $P\in \widetilde{C}$, the variety $\mathrm{V}^{\bm{a}}(C, \epsilon, P)$ is either empty, or smooth of dimension $\beta(g, \bm{a})$ at any $L\in \mathrm{V}^{\bm{a}}(C, \epsilon, P)$ such that $h^0(\widetilde{C}, L(-a_i \, P))=r+1-i$ for  $i=0,\dots, r$.
\end{thm}

This statement extends the result on the Prym-Brill-Noether varieties from \cite{welters1985theorem}. The proof  involves an infinitesimal argument and  the theory of limit linear series from Eisenbud-Harris \cite{MR723217}.
In conclusion, we deduce:

\begin{cor}
For a general curve $C$ of genus $g$, an arbitrary nontrivial $2$-torsion point $\epsilon$ in $\mathrm{Jac}(C)$, and a general point $P\in \widetilde{C}$, one has
\[
\mathrm{V}^{\bm{a}}(C, \epsilon, P)\neq \varnothing \iff \beta(g, \bm{a})\geq 0.
\]
When nonempty, $\mathrm{V}^{\bm{a}}(C, \epsilon, P)$ has dimension $\beta(g,\bm{a})$ and  class  equal to $\mathsf{B}(g,\bm{a})$ in $\mathrm{N}^*\left(\mathrm{Pr}^\pm (C, \epsilon), k\right)$ when \mbox{$\mathrm{char}(k) \neq 2$} and in $\mathrm{H}^*\left(\mathrm{Pr}^\pm (C, \epsilon), \mathbb{C}\right)$.
\end{cor}

For example, when $\beta(g, \bm{a})= 0$, the variety $\mathrm{V}^{\bm{a}}(C, \epsilon, P)$ for a general $(C, \epsilon, P)$ consists of finitely many distinct  points, counted by the degree of the class $\mathsf{B}(g,\bm{a})$:
\begin{equation*}
\deg \mathsf{B}(g,\bm{a}):= |\bm{a}|!\, 2^{|\bm{a}|-\ell(\bm{a})}\, \prod_{i=0}^r \frac{1}{a_i!}\, \prod_{0\leq j<i\leq r}\frac{a_i-a_j}{a_i+a_j} .
\end{equation*}
Here, the Poincar\'e formula yields $\deg \xi^{g-1}= (g-1)!=|\bm{a}|!$.
Remarkably, the above count has an interpretation as enumeration of certain tableaux:  comparing with \cite[Prop.~10.4]{zbMATH00051939} (see also Remarks \ref{rmk:mst} and \ref{rmk:naforunmarked}), we have
$\deg \mathsf{B}(g,\bm{a})=n_{\bm{a}}$ where
\begin{equation*}
n_{\bm{a}}:=  2^{|\bm{a}|-\ell(\bm{a})}\, \#\left\{ \mbox{standard shifted tableaux of shape $(a_r,\dots, a_0)$} \right\}.
\end{equation*}
Standard shifted tableaux are defined as follows. For $\bm{\lambda}=(\lambda_1, \dots, \lambda_\ell )$ with $\lambda_1  >\cdots >\lambda_\ell >0$, the \textit{shifted diagram} $S(\bm{\lambda})$ of shape $\bm{\lambda}$ is a configuration of $\ell$ rows of boxes, with $\lambda_i$ boxes in the $i$-th row, such that, for each $i>1$, the first box in row $i$ is placed underneath the second box in row $i-1$. A \textit{standard shifted tableau} of shape $\bm{\lambda}$ is a filling of the boxes of $S(\bm{\lambda})$ by the numbers $1,2,\dots, |\bm{\lambda}|$ such that the entries are strictly increasing down each column and from left to right along each row (as in Figure \ref{fig:T}).

\begin{figure}[htb]
\begin{center}
\ytableausetup{mathmode, boxsize=1.5em}
\begin{ytableau}
 1 & 2 & 4 & 6\\
\none & 3 & 5\\
\none & \none & 7 
\end{ytableau}
\end{center}
\caption{A standard shifted tableau of shape $(4,2,1)$.}
\label{fig:T}
\end{figure}

In the final part of the paper, we show that $n_{\bm{a}}$ has an incarnation also in the geometry of pointed Prym-Brill-Noether \textit{curves}:

\begin{thm}
\label{thm:expPT}
When $\dim \mathrm{V}^{\bm{a}}(C, \epsilon, P) = \beta(g, \bm{a})=1$,
the principally polarized abelian variety  $(\mathrm{Pr} (C,\epsilon), \Xi)$ is isomorphic  to a Prym-Tyurin variety for the curve $\mathrm{V}^{\bm{a}}(C, \epsilon, P)$ of exponent 
$
n_{\bm{a}}.
$  
\end{thm}

A polarized abelian variety $(A, \Xi)$ is said to be a \textit{Prym-Tyurin variety for a curve $X$} if there exists an embedding $\iota_A\colon A\hookrightarrow \mathrm{Jac}(X)$ such that 
the pull-back of the canonical  polarization $\Theta$ on $\mathrm{Jac}(X)$ is
$\iota_A^*[\Theta]=e\,[\Xi]$ for some integer $e$ called the \textit{exponent}  \cite[\S 12.2]{zbMATH02120946}.

Theorem \ref{thm:expPT} translates to the Prym-Brill-Noether setting a result from Ortega \cite{ortega2013brill} showing that for a general curve $C$ of odd genus,  $\mathrm{Jac}(C)$ is isomorphic as a principally polarized abelian variety to a Prym-Tyurin variety for the  Brill-Noether curve. 
Here, $(\mathrm{Jac}(C), \Theta)$ is replaced by $(\mathrm{Pr} (C,\epsilon), \Xi)$, and the Brill-Noether curve by the Prym-Brill-Noether curve.

\smallskip

It would be interesting to extend this study to include vanishing conditions at more fixed points. In the case of  Brill-Noether varieties of line bundles with prescribed vanishing at two points, a new determinantal formula in type A was introduced in \cite{act} to compute their classes. For this, the study of two-pointed Prym-Brill-Noether varieties could lead to the computation of a new degeneracy locus formula in type D.
Also, it would be interesting to compute the motivic invariants of Prym-Brill-Noether varieties, as done for the Brill-Noether varieties in \cite{pp, act, chan2021euler, anderson2020motivic}.

\smallskip

Theorem  \ref{thm:deglocus} is proved in \S\ref{sec:degstr}. We define a pointed Prym-Petri map in \S\ref{sec:pPPmap}, and prove a pointed analogue of the Prym-Petri theorem in \S\ref{sec:pPPThm}.  Theorem \ref{thm:mainintro} follows by combining Proposition \ref{prop:pPPimpliesexpdim} and Theorem \ref{thm:prympetri}. 
 Theorem \ref{thm:expPT} is proved in \S\ref{sec:expPT}.
Throughout, we work over a field of characteristic different from~$2$.


\section{The degeneracy locus structure}
\label{sec:degstr}

We define here the scheme structure for the pointed Prym-Brill-Noether varieties as a degeneracy locus in type D, generalizing the unpointed case from De Concini--Pragacz \cite{de1995class}, also reviewed in Fulton-Pragacz \cite[Ch.~8]{fulton2006schubert}. 
We focus  on the changes required in the pointed case, and refer to \cite{de1995class, fulton2006schubert} for the common aspects.
We prove Theorem  \ref{thm:deglocus} at the end of the section by applying a result from Anderson-Fulton \cite[\S4, Corollary]{anderson2018chern}.

Let $C$ be a genus $g$ curve, and $\epsilon$ a nontrivial {$2$-torsion} point in $\mathrm{Jac}(C)$. Let $\varphi\colon\widetilde{C}\rightarrow C$ be the irreducible \'etale double covering induced by~$\epsilon$. 
The maps used below are 
\[
\begin{tikzcd}
\mathrm{Pic}^{2g-2}(\widetilde{C}) \times \widetilde{C} \ar{d}{\widetilde{\delta}} \ar{r}[]{1\times \varphi}  & \mathrm{Pic}^{2g-2}(\widetilde{C}) \times C \ar{d}{{\delta}} \ar{r}[]{\gamma} & \mathrm{Pic}^{2g-2}(\widetilde{C}) \\
\widetilde{C}\ar{r}{\varphi} & C
\end{tikzcd}
\]
where $\gamma, \delta$, and $\widetilde{\delta}$ are the natural projections.

Let $\mathcal{L}$ be a Poincar\'e line bundle on $\mathrm{Pic}^{2g-2}(\widetilde{C}) \times \widetilde{C}$ 
appropriately normalized (specifically, as in \cite[pg.~93]{fulton2006schubert}).
Consider the  rank two bundle on $\mathrm{Pic}^{2g-2}(\widetilde{C}) \times C$:
\[
\mathcal{E} := (1\times \varphi)_*\, \mathcal{L}.
\]
A major insight used below is provided by a result of Mumford \cite{mumford1974prym}:  there exists a non-degenerate quadratic form 
\begin{equation}
\label{eq:qfE}
\mathcal{E}|_{\mathrm{Pr}^\pm (C,\epsilon)\times C} \longrightarrow \delta^*\omega_C|_{\mathrm{Pr}^\pm (C,\epsilon)\times C}.
\end{equation}

Select a divisor $D:=\sum_i P_i$ on $C$ with $\deg(D)>\!>0$ and $P_i\neq P_j$ for $i\neq j$.
Define
\[
\mathcal{V} := \gamma_* \left( \mathcal{E} (D) / \mathcal{E}(-D)\right) |_{\mathrm{Pr}^\pm (C,\epsilon)}, \qquad
\mathcal{U} := \gamma_* \left( \mathcal{E}  / \mathcal{E}(-D)\right) |_{\mathrm{Pr}^\pm (C,\epsilon)},
\]
where $\mathcal{E} (\pm D):= \mathcal{E}\, \otimes\, \delta^* \mathscr{O}_C(\pm D)$.
One has $\mathcal{U} \subset \mathcal{V}$ with $\mathrm{rank}\left( \mathcal{V} \right)=4\deg(D)$ and $\mathrm{rank}\left( \mathcal{U} \right)=2\deg(D)$.

For $\bm{a}=(0\leq a_0 < a_1< \dots< a_r)$ and a point $P$ in $\widetilde{C}$, define 
\[
\mathcal{E}_i := (1\times \varphi)_*\, \mathcal{L}(-a_{i}\,P)
\]
and
\[
\mathcal{W}_i := \gamma_* \left( \mathcal{E}_i(D)  \right) |_{\mathrm{Pr}^\pm (C,\epsilon)}
\]
for $i=0,\dots,r$. This gives a flag of  bundles 
\[
\mathcal{W}_r \subset \cdots \subset \mathcal{W}_{0} \subset \mathcal{V}
\]
with 
$\mathrm{rank}\left( \mathcal{W}_i  \right) = \chi\left( \mathcal{E}_i(D)\right) = 2\deg (D) - a_{i}$ for each $i$.

Select a point $L$ in $\mathrm{Pr}^\pm (C,\epsilon)$, and define $E:=\varphi_*\,  L$ and $E_i:=\varphi_*\,  L(-a_i\,P)$.
A construction of Mumford \cite{mumford1971theta} (see also \cite[App.~H]{fulton2006schubert}) shows  
\[
H^0( \widetilde{C}, L(-a_i\,P)) = H^0(C, E_i) = U \cap W_i  \subset V,
\]
where 
\[
W_i:=H^0(C,E_i(D)), \,\,\, U:= H^0(C,E/E(D)), \,\,\, V:=H^0(C,E(D)/E(-D)).
\]
Moreover, the non-degenerate quadratic form in \eqref{eq:qfE} induces a non-de\-gen\-er\-ate quadratic form on the vector space $V$, and its  subspaces $U$ and $W_r\subset \dots \subset W_0$ are all isotropic. 
As in \cite{de1995class, fulton2006schubert} the construction globalizes over $\mathrm{Pr}^\pm (C,\epsilon)$, hence 
\[
\gamma_* \,\mathcal{E}_i |_{\mathrm{Pr}^\pm (C,\epsilon)} = \mathcal{U} \cap \mathcal{W}_i  \subset \mathcal{V}
\]
and the  vector bundle $\mathcal{V}$ admits a non-degenerate quadratic form with values in $\mathscr{O}_{\mathrm{Pr}^\pm (C,\epsilon)}$ such that $\mathcal{U}$ and $\mathcal{W}_r\subset \dots \subset \mathcal{W}_0$ are all \textit{isotropic} subbundles. 

We define $\mathrm{V}^{\bm{a}}(C, \epsilon, P)$ as the locus in $\mathrm{Pr}^\pm (C,\epsilon)$ defined by the conditions
\[
\mathrm{rank}\left(\mathcal{W}_i \cap \mathcal{U}\right)\geq r+1-i, \qquad \mbox{for $i=0,\dots,r$}.
\]
This gives $\mathrm{V}^{\bm{a}}(C, \epsilon, P)$ the scheme structure induced by the Schubert variety corresponding to the partition $(a_r, a_{r-1}, \dots, a_0)$ inside  the orthogonal Grassmannian  of $2\deg(D)$-dimensional isotropic subspaces in a $4\deg(D)$-dimensional vector space. We are now ready for:

\begin{proof}[Proof of Theorem \ref{thm:deglocus}]
The first part of the statement follows from the degeneracy locus structure of $\mathrm{V}^{\bm{a}}(C, \epsilon, P)$ given above by means of \cite[Lemma, pg.~108]{fulton2006schubert}.
Consequently, a Pfaffian formula from \cite[\S4, Corollary]{anderson2018chern} in terms of the Chern classes of $\mathcal{U}$ and the $\mathcal{W}_i$'s, and depending on the partition $(a_r, a_{r-1}, \dots, a_0)$, yields a cohomology class supported on $\mathrm{V}^{\bm{a}}(C, \epsilon, P)$.
When $\dim (\mathrm{V}^{\bm{a}}(C, \epsilon, P))=\beta(g,\bm{a})$, the Pfaffian formula computes the class of  $\mathrm{V}^{\bm{a}}(C, \epsilon, P)$.

As in \cite[Lemma 5]{de1995class}, $\mathcal{U}$ has trivial Chern classes,  and $c(\mathcal{W}_i^\vee) = e^{2\xi}$ (since $\deg(D)>\!>0$, this is independent of $i$). 
The resulting Pfaffian formula for arbitrary $(a_r, a_{r-1}, \dots, a_0)$ coincides with a Pfaffian computed in \cite[Prop.~6]{de1995class} times $2^{-\ell(\bm{a})} (2\xi)^{|\bm{a}|}$.
This gives the class $\mathsf{B}(g,\bm{a})$ in \eqref{eq:classB}. 
\end{proof}

\begin{remark}
\label{rmk:mst}
Comparing with \cite[pg.~216]{zbMATH00051939},
 $n_{\bm{a}}$ can also be interpreted as the enumeration of marked shifted tableaux of shape $(a_r, a_{r-1}, \dots, a_0)$ with unmarked diagonal entries. Such tableaux are defined as follows. For $\bm{\lambda}=(\lambda_1, \dots, \lambda_\ell )$ with $\lambda_1  >\cdots >\lambda_\ell >0$, the \textit{marked shifted tableaux} of shape $\bm{\lambda}$ are the standard shifted tableaux of shape $\bm{\lambda}$ (as defined in the introduction) modified by marking some of the entries. Thus, the number of marked shifted tableaux of shape $\bm{\lambda}$ is
 $2^{|\bm{\lambda}|}\, \#\left\{ \mbox{standard shifted tableaux of shape $\bm{\lambda}$} \right\}$. The subset of those tableaux with \textit{unmarked diagonal entries} has size 
 \[
 2^{|\bm{\lambda}|-\ell}\, \#\left\{ \mbox{standard shifted tableaux of shape $\bm{\lambda}$} \right\}.
 \]
This equals   $n_{\bm{a}}$ for the shape $\bm{\lambda}$ given by $(a_r, a_{r-1}, \dots, a_0)$.
\end{remark}

\begin{remark}
\label{rmk:naforunmarked}
When $\bm{a}=(0,1,\dots, r)$, the number $n_{\bm{a}}$ equals the enumeration of standard Young tableaux of shape $(r, r-1, \dots, 1)$, that is, the number of filling of the configuration of $r$ left-aligned rows of boxes, with $r-i+1$ boxes in the $i$-th row, by the numbers $1,2,\dots, \frac{r(r+1)}{2}$, such that the entries are strictly increasing down each column and from left to right along each row. 
\end{remark}


\section{The pointed Prym-Petri map}
\label{sec:pPPmap}

Here we construct a  pointed Prym-Petri map, and  apply it to find an upper bound for the dimension of pointed Prym-Brill-Noether varieties.
The construction integrates the treatments of the unpointed case from \cite{welters1985theorem} and the  pointed Brill-Noether case from \cite{ciliberto1992endomorphisms}.

Let $C$ be a curve of genus $g$, 
and $\epsilon$ a non-trivial $2$-torsion point in $\mathrm{Jac}(C)$. Let $P\in \widetilde{C}$, where $\varphi \colon \widetilde{C} \rightarrow C$ is the irreducible \'etale double covering determined by $\epsilon$.
One has a splitting of the space of differential forms
\[
H^0\left(\widetilde{C}, \omega_{\widetilde{C}} \right) = H^0\left({C}, \omega_{{C}} \right) \oplus H^0\left({C}, \omega_{C}\otimes \epsilon \right)
\]
 into invariants and anti-invariant sections under the action of the covering involution~$\iota$.

The covering $\varphi$ induces a norm map $\mathrm{Nm}\colon \mathrm{Pic}(\widetilde{C})\rightarrow \mathrm{Pic}(C)$. This satisfies $\mathrm{Nm}\circ \varphi^* = 2 \,\mathrm{id}_{\mathrm{Pic}(C)}$ and $\varphi^* \circ \mathrm{Nm}  = \mathrm{id}_{\mathrm{Pic}(\widetilde{C})}\otimes  \iota^*$ \cite{mumford1974prym}.

For $\bm{a}=(0\leq a_0 < a_1< \dots< a_r)$ and $L\in \mathrm{V}^{\bm{a}}(C, \epsilon, P)$ such that $h^0(\widetilde{C}, L(-a_iP))=r+1-i$ for  $i=0,\dots, r$, select sections
\begin{align*}
&\sigma_i \in H^0( \widetilde{C}, L(-a_i \, P))\setminus H^0( \widetilde{C}, L(-a_{i+1} \, P))  \quad\mbox{ for } i=0,\dots, r-1, \\
&\sigma_r \in H^0( \widetilde{C},  L(-a_r \, P)).
\end{align*}
Let $M:=\iota^* L$. Since $\mathrm{Nm}(L)=\omega_C$, one has $\varphi^* \omega_C = L \otimes M$, and since $\varphi^* \omega_C =\omega_{\widetilde{C}}$, one deduces
$M = \omega_{\widetilde{C}}\otimes L^\vee$. 
Consider the composition 
\[
\bigoplus_{i=0}^r \, \langle \sigma_i \rangle  \otimes H^0(\widetilde{C}, M(a_i \, P)) \xrightarrow{\mu}  H^0\left(\widetilde{C}, \omega_{\widetilde{C}} \right)\xrightarrow{p}  H^0\left(C, \omega_C \otimes \epsilon\right)
\]
where $\mu$ is the Petri map obtained by product of sections, and $p$ is the projection onto the  anti-invariant subspace under the action of $\iota$.
Explicitly, the composition $p\circ \mu$ is given by
\[
\sigma\otimes \tau \mapsto  \frac{1}{2}\big(\sigma \tau - \iota(\sigma) \iota(\tau) \big).
\]
For each $i$, one has the following inclusions
\begin{equation}
\label{eq:iotaLiinMi}
\iota^* H^0( \widetilde{C}, L(-a_i \, P)) 
\hookrightarrow \iota^* H^0( \widetilde{C}, L) \xrightarrow{\cong} H^0(\widetilde{C}, M)
\hookrightarrow H^0(\widetilde{C}, M(a_i \, P)).
\end{equation}
In particular, the map $p\circ \mu$ restricts to a map
\[
\overline{\mu}\colon \bigoplus_{i=0}^r  \,\langle \sigma_i \rangle  \otimes H^0(\widetilde{C}, M(a_i \, P)) \Big/ \iota^* H^0( \widetilde{C}, L(-a_i \, P)) \longrightarrow  H^0(C, \omega_C \otimes \epsilon).
\]
Here the source of $\overline{\mu}$ is intended  as a subspace of the source of $p\circ \mu$. We call $\overline{\mu}$ the \textit{pointed Prym-Petri map} for $L$.

For instance, in the case $\bm{a}=(0,1,\dots, r)$ where no special vanishing is imposed at $P$, one has $h^0(\widetilde{C}, M(iP))=h^0(\widetilde{C}, M)$ for each $i$ by the Riemann-Roch formula. Hence the source of $\overline{\mu}$ is isomorphic to the image of $\wedge^2 \, H^0( \widetilde{C}, L)$ via the composition of 
\[
\wedge^2 \, H^0( \widetilde{C}, L) \hookrightarrow H^0( \widetilde{C}, L)^{\otimes 2} \xrightarrow{1\otimes \iota^*} H^0( \widetilde{C}, L) \otimes H^0( \widetilde{C}, M) 
\]
where the first map is given by $\sigma \wedge \tau \mapsto \frac{1}{2}(\sigma\otimes \tau - \tau\otimes \sigma)$. It follows that the map $\overline{\mu}$ recovers the Prym-Petri map in the unpointed case studied in \cite{welters1985theorem}.

\smallskip

Next, we show how the pointed Prym-Petri map encodes information on the tangent space of pointed Prym-Brill-Noether varieties. 
The variety $\mathrm{V}^{\bm{a}}(C, \epsilon, P)$ has a Zariski open subset
\[
\mathrm{V}^{\bm{a}}(C, \epsilon, P)^\circ := W^{\bm{a}}_{2g-2}(\widetilde{C}, P)^\circ \,\cap \, \mathrm{Pr}^\bullet (C,\epsilon)
\]
where $\bullet=+$ if $r$ is odd,  $\bullet=-$ if $r$ is even, and
\[
W^{\bm{a}}_{2g-2}(\widetilde{C}, P)^\circ := \left\{ L\in \mathrm{Pic}^{2g-2}(\widetilde{C}) \, : \, h^0(\widetilde{C}, L(-a_i \, P)) = r+1 -i,  \,\, \forall \, i \right\}.
\]
The scheme structure of $\mathrm{V}^{\bm{a}}(C, \epsilon, P)$ along $\mathrm{V}^{\bm{a}}(C, \epsilon, P)^\circ$ coincides with the natural scheme structure on the intersection 
 $W^{\bm{a}}_{2g-2}(\widetilde{C}, P)^\circ \,\cap \, \mathrm{Pr}^\bullet (C,\epsilon)$, as in \cite[Prop.~4(1)]{de1995class}.
 
 For  $L$   in $\mathrm{V}^{\bm{a}}(C, \epsilon, P)^\circ$, it follows that the tangent space of $\mathrm{V}^{\bm{a}}(C, \epsilon, P)$ at $L$  is the intersection of the tangent spaces of $W^{\bm{a}}_{2g-2}(\widetilde{C}, P)^\circ$ and $\mathrm{Pr}^\bullet (C,\epsilon)$ at $L$.
We proceed to identify these tangent spaces.

As in \cite[pg.~673]{welters1985theorem}, the tangent space of $\mathrm{Pr}^\bullet (C,\epsilon)$ at $L$ is
\begin{equation}
\label{eq:TPL}
T_{\mathrm{Pr}^\bullet (C,\epsilon)} (L) = H^0\left({C}, \omega_{C}\otimes \epsilon \right)^\vee \hookrightarrow H^0\left(\widetilde{C}, \omega_{\widetilde{C}} \right)^\vee
\end{equation}
where the inclusion is the transpose of the projection $p$. This follows from the fact that the tangent map of $\mathrm{Nm}$ at $L$ is canonically given by twice the transpose of the pull-back   $\varphi^*$:
\[
2 (\varphi^*)^t \colon H^0\left(\widetilde{C}, \omega_{\widetilde{C}} \right)^\vee \rightarrow H^0\left({C}, \omega_{C} \right)^\vee.
\]
Moreover, from \cite[Proof of (3.2)]{ciliberto1992endomorphisms}, the tangent space of $W^{\bm{a}}_{2g-2}(\widetilde{C}, P)^\circ$ at $L$ is
\begin{equation}
\label{eq:TWopen}
T_{W^{\bm{a}}_{2g-2}(\widetilde{C}, P)^\circ} (L) = \mathrm{Im}(\mu)^\perp \subset H^0\left(\widetilde{C}, \omega_{\widetilde{C}} \right)^\vee.
\end{equation}
Combining \eqref{eq:TPL} and \eqref{eq:TWopen}, we have 
\begin{align*}
T_{\mathrm{V}^{\bm{a}}(C, \epsilon, P)} (L) &= \mathrm{Im}(\mu)^\perp \cap H^0\left({C}, \omega_{C}\otimes \epsilon \right)^\vee\\
&= \mathrm{Im}(\overline{\mu})^\perp.
\end{align*}
To compute its dimension, note that since $h^0(\widetilde{C}, L(-a_i \, P))=r+1-i$ for $i=0,\dots, r$ and $M=\omega_{\widetilde{C}}\otimes L^\vee$,  one has by the Riemann-Roch formula:
\[
h^0(\widetilde{C}, M(a_i \, P)) = r+1-i+a_i, \qquad \mbox{for $i=0,\dots, r$}.
\]
It follows that the source of $\overline{\mu}$ has  dimension equal to $|\bm{a}|$, hence
\[
\dim \left(\mathrm{Im}(\overline{\mu})^\perp \right)= g-1 - |\bm{a}| + \dim \mathrm{Ker} (\overline{\mu}).
\]
This implies:

\begin{proposition}
\label{eq:upperbounddimV}
For  $L\in \mathrm{V}^{\bm{a}}(C, \epsilon, P)^\circ$, one has  
\[
\dim_L \left( \mathrm{V}^{\bm{a}}(C, \epsilon, P) \right) \leq \beta (g,\bm{a}) + \dim \mathrm{Ker} (\overline{\mu}).
\]
\end{proposition}

Comparing with the inequality in Theorem \ref{thm:deglocus}, the following condition emerges: 

\begin{definition}
A triple $(C,\epsilon, P)$ as above is said to \textit{satisfy the pointed Prym-Petri condition} if the map $\overline{\mu}$ is injective for all $L\in \mathrm{V}^{\bm{a}}(C, \epsilon, P)$ such that $h^0(\widetilde{C}, L(-a_i \, P))=r+1-i$  for  $i=0,\dots, r$.
\end{definition}

Combining the inequality in Theorem \ref{thm:deglocus} with Proposition \ref{eq:upperbounddimV}, we deduce: 

\begin{proposition}
\label{prop:pPPimpliesexpdim}
 If $(C,\epsilon, P)$ satisfies the pointed Prym-Petri condition, then 
 $\mathrm{V}^{\bm{a}}(C, \epsilon, P)$ is either empty, or smooth of dimension $\beta(g, \bm{a})$ at any $L\in \mathrm{V}^{\bm{a}}(C, \epsilon, P)$ such that $h^0(\widetilde{C}, L(-a_i \, P))=r+1-i$ for  $i=0,\dots, r$.
\end{proposition}


\section{The pointed Prym-Petri condition in the general case}
\label{sec:pPPThm}

\noindent  The main result of this section is the following \textit{pointed Prym-Petri Theorem}:

\begin{theorem}
\label{thm:prympetri}
For a general curve $C$ of genus $g$, an arbitrary nontrivial \mbox{$2$-torsion} point $\epsilon$ in $\mathrm{Jac}(C)$, and a general point $P\in \widetilde{C}$, 
the triple $(C,\epsilon, P)$ satisfies the pointed Prym-Petri condition.
\end{theorem}

Proposition \ref{prop:pPPimpliesexpdim} and Theorem \ref{thm:prympetri} directly imply Theorem \ref{thm:mainintro}. For the proof of Theorem \ref{thm:prympetri}, we adapt and refine ideas from the unpointed case treated in \cite{welters1985theorem} and the adjustment of the classical Brill-Noether case \cite{MR723217} to the pointed Brill-Noether case treated in \cite{ciliberto1992endomorphisms}.
First, we argue that:

\begin{lemma}
\label{lem:red2ex}
Exhibiting \textit{a single} triple $(C,\epsilon, P)$ for each genus  with $\epsilon\neq 0$ 
which satisfies the pointed Prym-Petri condition proves Theorem \ref{thm:prympetri}.
\end{lemma}

\begin{proof}
Satisfying the pointed Prym-Petri condition is an open condition on families of triples $(C, \epsilon, P)$, as in \cite[(2.1)]{welters1985theorem}. The statement follows from the irreducibility of the moduli space of triples $(C, \epsilon, P)$ of fixed genus  with~$\epsilon\neq 0$.
\end{proof}

For this, we consider a specific triple 
obtained as the geometric generic fiber of a family for each genus defined as follows.

\subsection{The family $(\mathscr{C}, \epsilon, P)$}
Let $T$ be the spectrum of a discrete valuation ring with parameter $t$,  closed point $0$, and generic point $\eta$. We will later use  the fact that $T$ has trivial Picard group.
Let $\pi\colon\mathscr{C}\rightarrow T$ be a flat projective family, with $\mathscr{C}$ a smooth surface, such that: 1) the generic fiber $\mathscr{C}_\eta$ is smooth and geometrically irreducible; and 2) the special fiber $\mathscr{C}_0$ is a reduced curve of arithmetic genus $g$ consisting of a chain of smooth components which are either rational or elliptic, and ordinary double points as only singularities (see Figure ~\ref{fig:}). Let $E_1,\dots, E_g$ be the elliptic components of $\mathscr{C}_0$. For each $i=1,\dots, g$, we assume furthermore that
$E_i$ meets the rest of the curve $\mathscr{C}_0$ in  two points $P_i$ and $Q_i$ such that $P_i-Q_i$ is not a torsion point in $\mathrm{Pic}(E_i)$.

We will repeatedly use the following property:
Given a dominant morphism $T'\rightarrow T$ of spectra of discrete valuation rings, the family \mbox{$\mathscr{C}'\rightarrow T'$} obtained by base extension and minimal resolution of singularities has special fiber $\mathscr{C}'_0$ satisfying the same assumptions as $\mathscr{C}_0$.

Up to extending the base, we can assume that there exists a  line bundle $\epsilon$ on $\mathscr{C}$ such that $\epsilon^2 \cong \mathscr{O}_{\mathscr{C}}$, with nontrivial restriction $\epsilon_\eta$ on $\mathscr{C}_\eta$, and with restriction to $\mathscr{C}_0$ nontrivial only over $E_g$. The line bundle $\epsilon$ gives rise to an irreducible \'etale double covering $\varphi\colon \widetilde{\mathscr{C}} \rightarrow \mathscr{C}$ over $T$, where \mbox{$\widetilde{\mathscr{C}}:=\mathrm{Spec}(\mathscr{O}_{\mathscr{C}}\oplus \epsilon)$,} the ring structure induced by $\epsilon^2 \cong \mathscr{O}_{\mathscr{C}}$. Let $\iota$ be the covering involution on $\widetilde{\mathscr{C}}$.
The  generic fiber $\widetilde{\mathscr{C}}_\eta$ is smooth and geometrically irreducible, and the special fiber $\widetilde{\mathscr{C}}_0$ is a reduced curve of arithmetic genus $2g-1$ consisting of a chain of smooth components which are either rational or elliptic, and ordinary double points as only singularities. The elliptic component can be denoted $E'_1, E''_1, \dots, E'_{g-1}, E''_{g-1}, \widetilde{E}_g$ such that for $i=1,\dots, g-1$, the restriction of the map $\varphi$ over $E_i$ is a reducible double covering $E'_i\sqcup E''_i \rightarrow E_i$, hence 
 $E'_i\cong E''_i \cong E_i$, and the restriction of the map $\varphi$ over $E_g$ is an irreducible double covering $\widetilde{E}_g\rightarrow E_g$ (see Fig.~\ref{fig:}). In particular, for each of the two points where $E_g$ meets the rest of the curve $\mathscr{C}_0$, the difference of the two preimages is a $2$-torsion point in $\mathrm{Jac}( \widetilde{E}_g)$.

Finally, we select a section $P\colon T \rightarrow \widetilde{\mathscr{C}}$ with the  property:  the corresponding point $P_0$ in $\widetilde{\mathscr{C}}_0$ lies in a rational component splitting the chain $\widetilde{\mathscr{C}}_0$ so that there is  a connected component of arithmetic genus $2g-1$ (as in Fig.~\ref{fig:}).

\begin{figure}[htb]

\begin{tikzpicture}

\draw[rotate=25] (-5.5,2.2) ellipse (1.cm and 0cm);

\node at (-4.5,-0.5) {\ldots}; 

\draw[rotate=-25] (-2.5,-1.5) ellipse (1.cm and 0cm);
\draw[rotate=25] (-2,0.5) ellipse (1.cm and 0cm);
\draw (0,0) ellipse (1.5cm and 1cm)
node[auto, below=0.7, label={0:$E''_1$}]{};
\node at (-1.5, -0.2) {$\bullet$};
\node at (-1.5, -1.8) {$P''_1$};
\draw[rotate=-25] (2.,0.5) ellipse (1.cm and 0cm);
\draw[rotate=25] (2.5,-1.5) ellipse (1.cm and 0cm);

\node at (4.5,-0.5) {\ldots};

\draw[rotate=-25] (5.5,2.2) ellipse (1.cm and 0cm);
\draw[rotate=25] (6,-3.2) ellipse (1.cm and 0cm);
\draw (8.5,0) ellipse (1.5cm and 1cm)
node[auto, below=0.7, label={0:$E''_{g-1}$}]{};
\node at (7.1, -0.3) {$\bullet$};
\node at (7, -1.8) {$P''_{g-1}$};
\draw[rotate=-25] (9.5,4) ellipse (1.cm and 0cm);
\draw[rotate=25] (10.,-5) ellipse (1.cm and 0cm);

\node at (12.7,-0.5) {\ldots};

\draw[rotate=-25] (13,5.7) ellipse (1.cm and 0cm);
\draw[rotate=25] (13.5,-6.8) ellipse (1.cm and 0cm);
\node at (15.5, -0.3) {$\bullet$};
\node at (15, -1.8) {$P''_{g}$};
   
\node at (20,-0.5) {\ldots};

 \begin{scope}[xshift=450, yshift=30, xscale=1.5, yscale=1.2]
\draw  (0,0) arc(-225:45:1) -- (0,{sqrt(2)}) arc (225:-45:1) -- cycle;
\end{scope}

\node at (18,-1.9) {$\widetilde{E}_g$};

\begin{scope}[xshift=0, yshift=100 ]

\node at (-6.2, -0.5) {$\bullet$};
\node at (-6.2, 1.) {$P_0$};

\draw[rotate=25] (-5.5,2.2) ellipse (1.cm and 0cm);

\node at (-4.5,-0.5) {\ldots}; 

\draw[rotate=-25] (-2.5,-1.5) ellipse (1.cm and 0cm);
\draw[rotate=25] (-2,0.5) ellipse (1.cm and 0cm);
\draw (0,0) ellipse (1.5cm and 1cm)
node[auto, above=0.7, label={0:$E'_{1}$}]{};
\node at (-1.5, -0.2) {$\bullet$};
\node at (-1.5, 1.8) {$P'_1$};
\draw[rotate=-25] (2.,0.5) ellipse (1.cm and 0cm);
\draw[rotate=25] (2.5,-1.5) ellipse (1.cm and 0cm);

\node at (4.5,-0.5) {\ldots};

\draw[rotate=-25] (5.5,2.2) ellipse (1.cm and 0cm);
\draw[rotate=25] (6,-3.2) ellipse (1.cm and 0cm);
\draw (8.5,0) ellipse (1.5cm and 1cm)
node[auto, above=0.7, label={0:$E'_{g-1}$}]{};
\node at (7.1, -0.3) {$\bullet$};
\node at (7, 1.8) {$P'_{g-1}$};
\draw[rotate=-25] (9.5,4) ellipse (1.cm and 0cm);
\draw[rotate=25] (10.,-5) ellipse (1.cm and 0cm);

\node at (12.7,-0.5) {\ldots};

\draw[rotate=-25] (13,5.7) ellipse (1.cm and 0cm);
\draw[rotate=25] (13.5,-6.8) ellipse (1.cm and 0cm);
\node at (15.5, -0.3) {$\bullet$};
\node at (15, 1.8) {$P'_{g}$};
   
\node at (20,-0.5) {\ldots}; 

\end{scope}

\end{tikzpicture}

$\downarrow$

\begin{tikzpicture}

\draw[rotate=25] (-5.5,2.2) ellipse (1.cm and 0cm);

\node at (-4.5,-0.5) {\ldots}; 

\draw[rotate=-25] (-2.5,-1.5) ellipse (1.cm and 0cm);
\draw[rotate=25] (-2,0.5) ellipse (1.cm and 0cm);
\draw (0,0) ellipse (1.5cm and 1cm)
node[auto, below=0.7, label={0:$E_{1}$}]{};
\node at (-1.5, -0.2) {$\bullet$};
\node at (-1.5, -1.8) {$P_1$};
\draw[rotate=-25] (2.,0.5) ellipse (1.cm and 0cm);
\draw[rotate=25] (2.5,-1.5) ellipse (1.cm and 0cm);

\node at (4.5,-0.5) {\ldots};

\draw[rotate=-25] (5.5,2.2) ellipse (1.cm and 0cm);
\draw[rotate=25] (6,-3.2) ellipse (1.cm and 0cm);
\draw (8.5,0) ellipse (1.5cm and 1cm)
node[auto, below=0.7, label={0:$E_{g-1}$}]{};
\node at (7.1, -0.3) {$\bullet$};
\node at (7, -1.8) {$P_{g-1}$};
\draw[rotate=-25] (9.5,4) ellipse (1.cm and 0cm);
\draw[rotate=25] (10.,-5) ellipse (1.cm and 0cm);

\node at (12.7,-0.5) {\ldots};

\draw[rotate=-25] (13,5.7) ellipse (1.cm and 0cm);
\draw[rotate=25] (13.5,-6.8) ellipse (1.cm and 0cm);
\draw (17,0) ellipse (1.5cm and 1cm)
node[auto, below=0.7, label={0:$E_{g}$}]{};
\node at (15.5, -0.3) {$\bullet$};
\node at (15, -1.8) {$P_{g}$};
   
\node at (20,-0.5) {\ldots}; 
   
\end{tikzpicture}  

    \caption{\small{A graphical representation of the curves ${\mathscr{C}}_0$, $\widetilde{\mathscr{C}}_0$, and the \'etale double covering $\widetilde{\mathscr{C}}_0\rightarrow {\mathscr{C}}_0$. Ellipses and  line segments stand for elliptic and rational components, respectively.}}
    \label{fig:}
\end{figure}
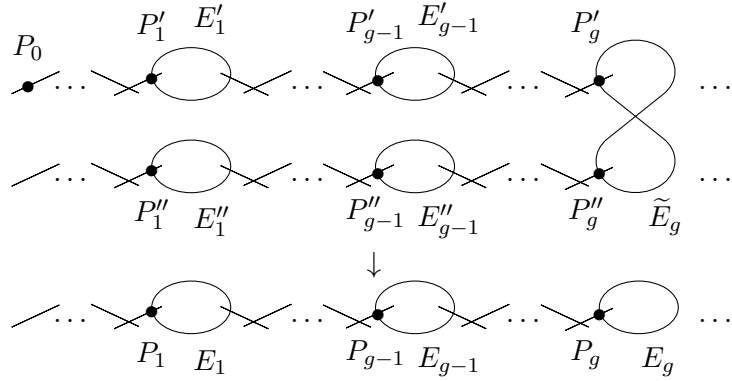

The maps are summarized in the following  diagram:
\[
\begin{tikzcd}
\arrow[loop left]{l}{\iota} \widetilde{\mathscr{C}} \ar{rr}{\varphi} \ar{dr}[]{\widetilde{\pi}} & &\mathscr{C}\ar{dl}{\pi}\\
& T \arrow[rightarrow, bend left=45]{ul}{P}
\end{tikzcd}
\]

\subsection{On the triple $(\mathscr{C}_\eta, \epsilon_\eta, P_\eta)$}
To show Theorem \ref{thm:prympetri}, it is enough to prove that the geometric generic fiber $(\mathscr{C}_{\overline{\eta}}, \epsilon_{\overline{\eta}}, P_{\overline{\eta}})$ of $(\mathscr{C}, \epsilon, P)$ satisfies the pointed Prym-Petri condition, where $\mathscr{C}_{\overline{\eta}}:= \mathscr{C}_\eta \otimes \overline{k(\eta)}$, $\epsilon_{\overline{\eta}} :=\epsilon_{\eta}\otimes \overline{k(\eta)}$, and $P_{\overline{\eta}}$ is the point in $\widetilde{\mathscr{C}}_{\overline{\eta}}:= \widetilde{\mathscr{C}}_\eta \otimes \overline{k(\eta)}$ induced by $P$.
As in \cite[pg.~272]{MR723217}, a line bundle  on $\mathscr{C}_{\overline{\eta}}$ comes from a line bundle defined over some finite extension of $k(\eta)$. 
After extending the base and changing the notation, it is enough to prove:

\begin{theorem}
\label{thm:prympetriateta}
The triple $(\mathscr{C}_\eta, \epsilon_\eta, P_\eta)$ satisfies the pointed Prym-Petri condition.
\end{theorem}

Theorem \ref{thm:prympetri} thus follows from Lemma \ref{lem:red2ex} and Theorem \ref{thm:prympetriateta}. The remainder of this section is dedicated to the proof of Theorem \ref{thm:prympetriateta}.

We proceed to describe the relevant pointed Prym-Petri map.
Let $\mathscr{L}_\eta$ in $\mathrm{V}^{\bm{a}}(\mathscr{C}_\eta, \epsilon_\eta, P_\eta)$ and $\mathscr{M}_\eta:=\iota^* \mathscr{L}_\eta$. Consider the sequence of sub--line bundles
\[
\mathscr{L}^r_\eta \hookrightarrow \cdots \hookrightarrow \mathscr{L}^0_\eta \hookrightarrow \mathscr{L}_\eta, \qquad
\mathscr{M}_\eta \hookrightarrow \mathscr{M}^0_\eta \hookrightarrow \cdots \hookrightarrow \mathscr{M}^r_\eta,
\]
where
\begin{equation}
\label{eq:LietaMieta}
\mathscr{L}^i_\eta := \mathscr{L}_\eta(-a_i \, P_\eta), \qquad
\mathscr{M}^i_\eta:= \mathscr{M}_\eta(a_i \, P_\eta).
\end{equation}
Assume $h^0(\widetilde{\mathscr{C}}_\eta, \mathscr{L}^i_\eta)=r+1-i$ for each $i$, and select sections 
\begin{align}
\label{eq:condsigmai}
& \sigma_i\in \widetilde{\pi}_*\, \mathscr{L}^i_\eta \setminus \widetilde{\pi}_*\, \mathscr{L}^{i+1}_\eta \quad\mbox{ for } i=0,\dots, r-1, 
& \sigma_r \in \widetilde{\pi}_*\, \mathscr{L}^{r}_\eta.
\end{align}
In particular, $\{\sigma_i\}_i$ is a basis of $\widetilde{\pi}_*\, \mathscr{L}^0_\eta$.
For each $i$, 
consider the composition of the inclusions
\begin{equation}
\label{eq:iotaLietainMieta}
\iota^* \, \widetilde{\pi}_*\, \mathscr{L}^i_\eta 
\hookrightarrow \iota^* \, \widetilde{\pi}_*\, \mathscr{L}_\eta \xrightarrow{\cong} \widetilde{\pi}_*\, \mathscr{M}_\eta
\hookrightarrow \widetilde{\pi}_*\, \mathscr{M}^i_\eta
\end{equation}
as in \eqref{eq:iotaLiinMi}.
The pointed Prym-Petri map for $\mathscr{L}_\eta$ is defined as in \S\ref{sec:pPPmap}:
\begin{equation}
\label{eq:mubareta}
\overline{\mu}_\eta \colon \bigoplus_{i=0}^r  \,\langle \sigma_i \rangle  \otimes \widetilde{\pi}_*\, \mathscr{M}^i_\eta \Big/ \iota^* \, \widetilde{\pi}_*\, \mathscr{L}^i_\eta \longrightarrow \widetilde{\pi}_*\left( \omega_{\mathscr{C}_\eta} \otimes \epsilon_\eta \right).
\end{equation}

\subsection{Extending $\overline{\mu}_\eta$ over $T$}
First, we extend the map $\overline{\mu}_\eta$ to a map over $T$.
The source of the map $\overline{\mu}_\eta$  is a vector subspace of 
\[
\bigoplus_{i=0}^r  \,\langle \sigma_i \rangle  \otimes \widetilde{\pi}_*\, \mathscr{M}^i_\eta \subseteq \bigoplus_{i=0}^r  \,  \widetilde{\pi}_*\, \mathscr{L}^i_\eta \otimes \widetilde{\pi}_*\, \mathscr{M}^i_\eta.
\] 
Thus we proceed by extending the line bundles $\mathscr{L}_\eta$, $\mathscr{L}^i_\eta$, and $\mathscr{M}^i_\eta$ over $\widetilde{\mathscr{C}}$.

After base extension and minimally resolving the singularities, the line bundle $\mathscr{L}_\eta$ on $\widetilde{\mathscr{C}}_\eta$ extends to a line bundle $\mathscr{L}$ on $\widetilde{\mathscr{C}}$. 
Since $\mathrm{Nm}(\mathscr{L}_\eta)\cong \omega_{\mathscr{C}_\eta}$,
one has  $\mathrm{Nm}(\mathscr{L})\cong \omega_{\mathscr{C}/T} \otimes \mathscr{F}$, where $\mathscr{F}$ is a line bundle assigned to a linear combination of the components of $\mathscr{C}_0$. Since $T$ has trivial Picard group,  $\mathscr{O}_{\mathscr{C}}(\mathscr{C}_0)\cong\mathscr{O}_{\mathscr{C}}$, hence we can assume that $E_g$ does not appear in the linear combination defining $\mathscr{F}$. Moreover, since $\mathrm{Nm}(\mathscr{O}_{\widetilde{\mathscr{C}}}\,(E'_i))\cong\mathrm{Nm}(\mathscr{O}_{\widetilde{\mathscr{C}}}\,(E''_i))\cong\mathscr{O}_{{\mathscr{C}}}\,(E_i)$ for $i=1,\dots,g-1$, and similarly for all rational components, we can assume $\mathrm{Nm}(\mathscr{L})\cong \omega_{\mathscr{C}/T}$ after an appropriate twist.

By the theory of limit linear series \cite{MR723217}, for every component $Y$ of $\widetilde{\mathscr{C}}_0$, there exists a line bundle $\mathscr{L}_Y$ extending $\mathscr{L}_\eta$ over $\widetilde{\mathscr{C}}$  which has degree~$0$ on all components of $\widetilde{\mathscr{C}}_0$ except $Y$. The line bundle $\mathscr{L}_Y$ is obtained from $\mathscr{L}$ by twisting with a linear combination of the components of $\widetilde{\mathscr{C}}_0$. Similarly, one has  extensions over ${\mathscr{C}}$ of line bundles on ${\mathscr{C}}_\eta$.
As noted in \cite[pg.~677]{welters1985theorem}, one has $\iota^*(\mathscr{L}_Y)\cong (\iota^*\mathscr{L})_{\iota(Y)}$ and $\mathrm{Nm}(\mathscr{L}_Y)\cong \mathrm{Nm}(\mathscr{L})_{\varphi(Y)}$.

For each $i$, consider the extension $\mathscr{L}^i_{\widetilde{E}_g}$ of $\mathscr{L}^i_\eta$ over $\widetilde{\mathscr{C}}$ which has degree~$0$ on all components of $\widetilde{\mathscr{C}}_0$ except $\widetilde{E}_g$, and similarly: the extension $\mathscr{M}^i_{\widetilde{E}_g}$ of~$\mathscr{M}^i_\eta$;
the extension $\mathscr{L}_{\widetilde{E}_g}$ of $\mathscr{L}_\eta$; and the extension $\mathscr{M}_{\widetilde{E}_g}$ of $\mathscr{M}_\eta$.
We identify $\widetilde{\pi}_*\, \mathscr{L}^i_{\widetilde{E}_g}$ with the $\mathscr{O}_T$-submodule of $\widetilde{\pi}_*\, \mathscr{L}^i_\eta$ which spans this space, and similarly for $\widetilde{\pi}_*\, \mathscr{M}^i_{\widetilde{E}_g}$. Thus one has inclusions of $\mathscr{O}_T$-submodules 
\[
\widetilde{\pi}_*\, \mathscr{L}^i_{\widetilde{E}_g} \hookrightarrow \widetilde{\pi}_*\, \mathscr{L}^i_\eta \qquad\mbox{and}\qquad 
\widetilde{\pi}_*\, \mathscr{M}^i_{\widetilde{E}_g} \hookrightarrow \widetilde{\pi}_*\, \mathscr{M}^i_\eta.
\] 
After possibly replacing the basis $\{\sigma_i\}_i$ of $\widetilde{\pi}_* \, \mathscr{L}^0_\eta$, we can assume in addition to \eqref{eq:condsigmai} that there exist integers $\alpha_i$ such that 
\begin{align*}
& t^{\alpha_i}{\sigma}_i \in \widetilde{\pi}_*\, \mathscr{L}^i_{\widetilde{E}_g} \setminus \widetilde{\pi}_*\, \mathscr{L}^{i+1}_{\widetilde{E}_g} \quad\mbox{ for } i=0,\dots, r-1, 
& t^{\alpha_r}{\sigma}_r \in \widetilde{\pi}_*\, \mathscr{L}^{r}_{\widetilde{E}_g}
\end{align*}
as in \cite[1.2]{MR723217}.
In particular, $\{t^{\alpha_i}\sigma_i\}_i$ is a basis of $\widetilde{\pi}_*\, \mathscr{L}^0_{\widetilde{E}_g}$. To simplify the notation, let 
\[
\widehat{\sigma}_i:= t^{\alpha_i}\sigma_i \qquad\mbox{ for } i=0,\dots, r.
\]

The map
\begin{equation*}
\overline{\mu}\colon \bigoplus_{i=0}^r  \,\langle \widehat{\sigma}_i \rangle  \otimes \widetilde{\pi}_*\, \mathscr{M}^i_{\widetilde{E}_g} \Big/ \iota^* \, \widetilde{\pi}_*\, \mathscr{L}^i_{\widetilde{E}_g} \longrightarrow \widetilde{\pi}_*\left( \omega_{\mathscr{C}} \otimes \epsilon \right)
\end{equation*}
induced by product of sections gives \eqref{eq:mubareta} over $\eta$. Here, the inclusion of $\mathscr{O}_T$-submodules 
\begin{equation}
\label{eq:iotaLiEginMiEg}
\iota^* \, \widetilde{\pi}_*\, \mathscr{L}^i_{\widetilde{E}_g}
\hookrightarrow \widetilde{\pi}_*\, \mathscr{M}^i_{\widetilde{E}_g}
\end{equation}
is induced from \eqref{eq:iotaLietainMieta}.

\subsection{On the kernel of $\overline{\mu}_\eta$}
Let 
$\rho \in \mathrm{Ker}\left( \overline{\mu}_\eta\right)$. 
We will show  $\rho = 0$.
Assuming $\rho\neq 0$, there exists a unique integer $\gamma$ such that 
\[
t^\gamma\rho \in 
\bigoplus_{i=0}^r  \left( \langle \widehat{\sigma}_i\rangle \otimes \widetilde{\pi}_*\, \mathscr{M}^i_{\widetilde{E}_g} 
\setminus t \left( \langle \widehat{\sigma}_i\rangle \otimes \widetilde{\pi}_*\, \mathscr{M}^i_{\widetilde{E}_g} \right) \right).
\]
Since $\rho \in \mathrm{Ker}\left( \overline{\mu}_\eta\right)$,
one has
$t^\gamma\rho \in \mathrm{Ker}\left( \overline{\mu}\right).$
Next, for each $i$, we restrict to 
\begin{align}
\label{eq:LiMi}
\begin{split}
L^i &:= \mathrm{Im}\left( \widetilde{\pi}_*\, \mathscr{L}^i_{\widetilde{E}_g}  \rightarrow H^0 \left( \mathscr{L}^i_{\widetilde{E}_g} \otimes \mathscr{O}_{\widetilde{E}_g} \right) \right), \\
M^i  &:= \mathrm{Im}\left( \widetilde{\pi}_*\, \mathscr{M}^i_{\widetilde{E}_g}  \rightarrow H^0 \left( \mathscr{M}^i_{\widetilde{E}_g} \otimes \mathscr{O}_{\widetilde{E}_g} \right) \right).
\end{split}
\end{align}
Let $\overline{\sigma}_i\in L^i$ be the image of $\widehat{\sigma}_i$ via the restriction map 
\mbox{$
\widetilde{\pi}_*\, \mathscr{L}^i_{\widetilde{E}_g} \rightarrow L^i.
$}
One has $L^i=\langle \overline{\sigma}_i, \dots, \overline{\sigma}_r\rangle$.
Similarly, let $\overline{\rho}$ be the image of $t^\gamma\rho$ via the map 
\[
\bigoplus_{i=0}^r  \langle \widehat{\sigma}_i\rangle \otimes \widetilde{\pi}_*\, \mathscr{M}^i_{\widetilde{E}_g} \rightarrow 
\bigoplus_{i=0}^r  \langle \overline{\sigma}_i\rangle \otimes M^i.
\]
By assumption, we have  
\begin{equation}
\label{eq:nonzeroassumption}
\overline{\rho} \neq 0 \quad\mbox{ in }\quad \bigoplus_{i=0}^r  \,\langle \overline{\sigma}_i \rangle  \otimes 
M^i \big/ \iota^* L^i.
\end{equation}
Here, the inclusion $\iota^* L^i \hookrightarrow M^i$ is induced from \eqref{eq:iotaLiEginMiEg}.

Let $P'_g$ be the point where $\widetilde{E}_g$ meets the connected components in the rest of $\widetilde{\mathscr{C}}_0$
which contains the point $P_0$, and let $P''_g:=\iota(P'_g)$ (as in Fig.~\ref{fig:}).

\begin{lemma}
\label{lem:ordPP}
The assumption 
$\rho \in \mathrm{Ker}\left( \overline{\mu}_\eta\right)$
implies  
 \begin{equation}
 \label{eq:ordPP}
 \mathrm{ord}_{P'_g}\left(\overline{\rho}\right) \geq 2g-2\qquad\mbox{and}\qquad
  \mathrm{ord}_{P''_g}\left(\overline{\rho} \right) \geq 2g-2.
 \end{equation}
 \end{lemma}
 This means that $\overline{\rho}$ is a linear combination of elements $\overline{\sigma}_i\otimes \overline{\tau}_j$ with $\overline{\tau}_j\in M^i$ such that 
 $\mathrm{ord}_{P'_g}(\overline{\sigma}_i)+\mathrm{ord}_{P'_g}(\overline{\tau}_j)\geq 2g-2$ for all $i,j$, and similarly at $P''_g$.

In the proof of the lemma, we combine  results from \cite[pp.~278-280]{MR723217}, \cite[Proof of 3.2]{ciliberto1992endomorphisms}, and \cite[p.~679]{welters1985theorem}.

\begin{proof}
First we set the notation.
Let $k\in \{1, \dots, g-1\}$.
We can assume that there exist integers $\alpha_{i, k}$ such that 
\begin{align*}
& t^{\alpha_{i,k}}{\sigma}_i \in \widetilde{\pi}_*\, \mathscr{L}^i_{E'_k} \setminus \widetilde{\pi}_*\, \mathscr{L}^{i+1}_{E'_k} \quad\mbox{ for } i=0,\dots, r-1, 
& t^{\alpha_{r, k}}{\sigma}_r \in \widetilde{\pi}_*\, \mathscr{L}^{r}_{E'_k}
\end{align*}
as in \cite[1.2]{MR723217}.
Let $\widehat{\sigma}_{i, k}:= t^{\alpha_{i, k}}\sigma_i$ for  $i=0,\dots, r$.
Then there exists a unique integer $\gamma_k$ such that
\[
t^{\gamma_k}\rho \in 
\bigoplus_{i=0}^r  \left( \langle \widehat{\sigma}_{i, k}\rangle \otimes \widetilde{\pi}_*\, \mathscr{M}^i_{E'_k} 
\setminus t \left( \langle \widehat{\sigma}_{i, k}\rangle \otimes \widetilde{\pi}_*\, \mathscr{M}^i_{E'_k} \right) \right).
\]
When $k=g$, we set $E'_{g}:= \widetilde{E}_g$, $\widehat{\sigma}_{i, g}:=\widehat{\sigma}_i$, and $\gamma_g :=\gamma$. 

For $k=1,\dots,g-1$, let  $P'_k$ be the point where $E'_{k}$ meets the connected components in the rest of $\widetilde{\mathscr{C}}_0$ which contains the point $P_0$ (as in Fig.~\ref{fig:}).
From  \cite[Proof of 3.2]{ciliberto1992endomorphisms}, which adapts to the pointed case  the argument first treated in the unpointed case in \cite[pp.~278-280]{MR723217}, 
the assumption $\rho \in \mathrm{Ker}\left( \overline{\mu}_\eta\right)$ 
implies   
\begin{equation}
\label{eq:EHineq}
 \mathrm{ord}_{P'_{k+1}}\left(t^{\gamma_{k+1}}\rho |_{E'_{k+1}}\right) \geq 
 \mathrm{ord}_{P'_{k}}\left(t^{\gamma_{k}}\rho |_{E'_{k}}\right) + 2
\quad \mbox{for $k=1,\dots,g-1$.}
\end{equation}
Here one uses that the difference of the two nodal points in $E'_k$ is not a torsion point in $\mathrm{Jac}(E'_k)$ for $k=1,\dots,g-1$.
Actually, the proof in \cite{MR723217, ciliberto1992endomorphisms}  uses families of curves with special fiber 
obtained from $\widetilde{\mathscr{C}}_0$ by replacing each elliptic component with a rational component meeting the rest of the curve in the same two points and meeting in a third point a new chain of rational curves ending with an elliptic component.
However, as it was first shown in \cite{welters1985theorem}, the argument extends to the case of families of curves with special fiber given by a chain of rational and elliptic curves, and gives \eqref{eq:EHineq} for a given $k$, as long as the difference of the two nodal points in $E'_k$ is not a torsion point in $\mathrm{Jac}(E'_k)$. 
Moreover, such families  allow one to work more generally over an arbitrary field of characteristic different from $2$ \cite{welters1985theorem}.

The inequality \eqref{eq:EHineq} for $k=1,\dots,g-1$ implies $\mathrm{ord}_{P'_g}\left(\overline{\rho}\right) \geq 2g-2$.
A similar argument holds after replacing $P'_a$ with $P''_a:=\iota(P'_a)$, and $E'_a$ with $E''_a:=\iota(E'_a)$, for each $a=1,\dots, g$, hence the statement.
\end{proof}

We emphasize that the difference $P'_g - P''_g$  is a \mbox{$2$-torsion} point in $\mathrm{Jac}(\widetilde{E}_g)$, hence  the above application of the argument from \cite{MR723217, ciliberto1992endomorphisms} does not yield anything more than \eqref{eq:ordPP}.

While the conclusion of Lemma \ref{lem:ordPP} is similar to \cite[(2.20)]{welters1985theorem}, in the next step  we need to account for the additional complexity  
given by the more complicated source of $\overline{\mu}$.

\subsection{A vanishing statement}
For $i=0,\dots,r$, 
one has from \eqref{eq:LietaMieta}:
\begin{align}
\label{eq:LiEg}
\begin{split}
\mathscr{L}^i_{\widetilde{E}_g} \otimes \mathscr{O}_{\widetilde{E}_g} 
&\cong \mathscr{L}_{\widetilde{E}_g} \otimes \mathscr{O}_{\widetilde{E}_g}\left(-a_i \, P'_g\right),\\
\mathscr{M}^i_{\widetilde{E}_g} \otimes \mathscr{O}_{\widetilde{E}_g}
& \cong \mathscr{M}_{\widetilde{E}_g} \otimes \mathscr{O}_{\widetilde{E}_g}\left(a_i \, P'_g\right).
\end{split}
\end{align}

From \cite[(2.21)]{welters1985theorem}, the line bundles $\mathscr{L}_{\widetilde{E}_g} \otimes \mathscr{O}_{\widetilde{E}_g}$ and $\mathscr{M}_{\widetilde{E}_g} \otimes \mathscr{O}_{\widetilde{E}_g}$ are both isomorphic to 
\begin{equation}
\label{eq:LEgOEg}
\mbox{either }\quad \mathscr{O}_{\widetilde{E}_g}\left((2g-2)P'_g\right) \quad\mbox{ or }\quad \mathscr{O}_{\widetilde{E}_g}\left((2g-3)P'_g + P''_g\right).
\end{equation}
This follows from the isomorphisms
\begin{align*}
\mathrm{Nm}\left( \mathscr{L}_{\widetilde{E}_g} \otimes \mathscr{O}_{\widetilde{E}_g} \right) &\cong \mathscr{O}_{E_g}\left( (2g-2)P_g\right), \\
\mathrm{Nm}\left( \mathscr{M}_{\widetilde{E}_g} \otimes \mathscr{O}_{\widetilde{E}_g} \right) &\cong \mathscr{O}_{E_g}\left( (2g-2)P_g\right),
\end{align*}
and the fact that the line bundles in \eqref{eq:LEgOEg} are the  two distinct line bundles in the inverse image of $\mathscr{O}_{E_g}\left( (2g-2)P_g\right)$ via the $(2:1)$ norm map 
\[
\mathrm{Nm} \colon \mathrm{Pic}^{2g-2}(\widetilde{E}_g)\rightarrow \mathrm{Pic}^{2g-2}({E}_g)
\]
and are both invariant by $\iota$.
For this, one uses that $P'_g - P''_g$ is a $2$-torsion point in $\mathrm{Jac}( \widetilde{E}_g)$.
Combining  with \eqref{eq:LiEg}, we deduce the following:

\begin{lemma}
\label{lem:zeroconclusion}
Let $\mathscr{L}_{\widetilde{E}_g} \otimes \mathscr{O}_{\widetilde{E}_g}$ and $\mathscr{M}_{\widetilde{E}_g} \otimes \mathscr{O}_{\widetilde{E}_g}$ be both isomorphic to 
either one of the two line bundles in \eqref{eq:LEgOEg} and assume
\[
\overline{\rho} \in \bigoplus_{i=0}^r  \,\langle \overline{\sigma}_i \rangle  \otimes M^i \big/ \iota^* L^i
\]
satisfies \eqref{eq:ordPP}.
Then 
$\overline{\rho} = 0.$
\end{lemma}

\begin{proof}
From \eqref{eq:LiMi}, \eqref{eq:LiEg}, and the assumption $\mathscr{L}_{\widetilde{E}_g} \otimes \mathscr{O}_{\widetilde{E}_g}\cong\mathscr{M}_{\widetilde{E}_g} \otimes \mathscr{O}_{\widetilde{E}_g}$, all spaces $L^0, \dots, L^r$ and $M^0,\dots, M^r$ inject in $V:= H^0( \mathscr{M}_{\widetilde{E}_g} \otimes \mathscr{O}_{\widetilde{E}_g}(a_r \, P'_g) )$.
We select a basis $\{e_n\}_n$ of $V$ such that the vectors $e_n$ have distinct orders of vanishing at $P'_g$ and distinct orders of vanishing at $P''_g$.

For this,  consider first the case $\mathscr{L}_{\widetilde{E}_g} \otimes \mathscr{O}_{\widetilde{E}_g} \cong \mathscr{M}_{\widetilde{E}_g} \otimes \mathscr{O}_{\widetilde{E}_g} \cong \mathscr{O}_{\widetilde{E}_g}\left((2g-2)P'_g\right)$.
Recall that $2P'_g \equiv 2P''_g$ on $\widetilde{E}_g$. 
Let $Q$ and $R$ be points in $\widetilde{E}_g\setminus \{P'_g, P''_g\}$ such that $P'_g + P''_g \equiv Q+R$. 
Define the following divisors on $\widetilde{E}_g$
\begin{align}
\label{eq:Divs4es}
\begin{split}
D_{2k}:= \,\,& (2g-2+a_r -2k)P'_g + 2k P''_g \\ 
&\qquad \mbox{for } k=0,\dots,g-1+\Big\lfloor\frac{a_r}{2}\Big\rfloor,\\
D_{2k+1}:= \,\, & (2g-5+ a_r-2k)P'_g + (2k+1) P''_g + Q+R \\ 
&\qquad \mbox{for } k=0,\dots,g+\Big\lfloor\frac{a_r-5}{2}\Big\rfloor.
\end{split}
\end{align}
For $n\in\{0,\dots,2g-4+a_r, 2g-2+a_r\}$, let $e_n$ be a section in $V$ with divisor $D_n$.
The sequence of orders of vanishing of such  $2g-2+a_r$  sections $e_n$ at  $P'_g$ is $(0,\dots,2g-4+a_r, 2g-2+a_r)$, and similarly at $P''_g$. It follows that $\{e_n\}_n$ is a basis of $V$ with the desired property.

Write $\overline{\rho}=\oplus_{i=0}^r \,\overline{\rho}_i$ with $\overline{\rho}_i\in \langle \overline{\sigma}_i \rangle  \otimes M^i$  for each $i$. 
Select $i$ such that $\overline{\rho}_i\neq 0$.
The assumption \eqref{eq:ordPP} implies that 
\begin{equation}
\label{eq:ordPPi}
 \mathrm{ord}_{P'_g}\left(\overline{\rho}_i\right) \geq 2g-2\qquad\mbox{and}\qquad
  \mathrm{ord}_{P''_g}\left(\overline{\rho}_i \right) \geq 2g-2.
\end{equation}
The basis $\{e_n\}_n$ of $V$ induces 
\begin{align*}
&\mbox{a basis } \{\lambda_\ell\}_\ell \mbox{ of the subspace } H^0\left(\mathscr{L}_{\widetilde{E}_g} \otimes \mathscr{O}_{\widetilde{E}_g}(-a_i \, P'_g)\right), \mbox{ and} \\
&\mbox{a basis } \{\mu_m\}_m \mbox{ of the subspace } H^0\left(\mathscr{M}_{\widetilde{E}_g} \otimes \mathscr{O}_{\widetilde{E}_g}(a_i \, P'_g)\right).
\end{align*}
Let $c_\ell:=\mathrm{ord}_{P'_g}(\lambda_\ell)$. The construction of the basis via the divisors in \eqref{eq:Divs4es} implies
\begin{equation}
\label{eq:ordP''la}
\mathrm{ord}_{P''_g}(\lambda_\ell) = \left\{
\begin{array}{ll}
2g-2-a_i-c_\ell & \quad\mbox{ if}\,\,\, a_i + c_\ell \equiv 0 \,\,\,\mathrm{mod} \,\,2,\\
2g-4-a_i-c_\ell & \quad\mbox{ if}\,\,\, a_i + c_\ell \equiv 1 \,\,\,\mathrm{mod} \,\,2.
\end{array}
\right. 
\end{equation}
Similarly, let $b_m:=\mathrm{ord}_{P'_g}(\mu_m)$. Then 
\begin{equation}
\label{eq:ordP''mu}
\mathrm{ord}_{P''_g}(\mu_m) = \left\{
\begin{array}{ll}
2g-2+a_i-b_m & \quad\mbox{ if}\,\,\, a_i + b_m \equiv 0 \,\,\,\mathrm{mod} \,\,2,\\
2g-4+a_i-b_m & \quad\mbox{ if}\,\,\, a_i + b_m \equiv 1 \,\,\,\mathrm{mod} \,\,2.
\end{array}
\right. 
\end{equation}
Since $\langle \overline{\sigma}_i \rangle \subseteq H^0(\mathscr{L}_{\widetilde{E}_g} \otimes \mathscr{O}_{\widetilde{E}_g}(-a_i \, P'_g))$ and 
$M^i\subseteq H^0(\mathscr{M}_{\widetilde{E}_g} \otimes \mathscr{O}_{\widetilde{E}_g}(a_i \, P'_g))$,
we can then write $\overline{\rho}_i = \sum_{\ell,m} d_{\ell,m}\, \lambda_\ell \otimes \mu_m$, for some coefficients $d_{\ell,m}$.
From \eqref{eq:ordPPi}, and since the basis $\{\lambda_\ell\}_\ell$ (respectively, $\{\mu_m\}_m$) has distinct orders of vanishing at $P'_g$, and similarly at $P''_g$,
one has $d_{\ell,m}=0$ for those $\ell,m$ which fail the conditions
\[
\mathrm{ord}_{P'_g}(\lambda_\ell)+\mathrm{ord}_{P'_g}(\mu_m)\geq 2g-2 \,\,\mbox{ and }\,\,
\mathrm{ord}_{P''_g}(\lambda_\ell)+\mathrm{ord}_{P''_g}(\mu_m)\geq 2g-2.
\]
Furthermore, these conditions 
imply first $c_\ell+b_m\geq 2g-2$, and then
\begin{equation}
\label{eq:conclusion}
c_\ell+b_m= 2g-2 \qquad\mbox{and} \qquad a_i + c_\ell \equiv  a_i + b_m \equiv 0 \,\,\,\mathrm{mod} \,\,2.
\end{equation}
It follows that 
\[
\mathrm{div}\left(\lambda_\ell \right) = c_\ell \,P'_g + (b_m-a_i) P''_g \qquad \mbox{and} \qquad
\mathrm{div}\left(\mu_m \right) = b_m \, P'_g + (a_i+c_\ell) P''_g.
\]
Hence, the image of $\iota^* \lambda_\ell$ via the composition of the inclusions 
\begin{multline*}
\iota^* H^0\left(\mathscr{L}_{\widetilde{E}_g} \otimes \mathscr{O}_{\widetilde{E}_g}(-a_i \, P'_g)\right)
\hookrightarrow \iota^* H^0\left(\mathscr{L}_{\widetilde{E}_g} \otimes \mathscr{O}_{\widetilde{E}_g}\right)
\xrightarrow{\cong} H^0\left(\mathscr{M}_{\widetilde{E}_g} \otimes \mathscr{O}_{\widetilde{E}_g}\right)\\
\hookrightarrow H^0\left(\mathscr{M}_{\widetilde{E}_g} \otimes \mathscr{O}_{\widetilde{E}_g}(a_i \, P'_g)\right)
\end{multline*}
lies in $\langle\mu_m\rangle$. We deduce  $\overline{\rho}_i\in \langle \overline{\sigma}_i \rangle \otimes \mathrm{Im}\left(\iota^* L^i \hookrightarrow M^i\right)$, hence the statement.

\smallskip

Finally, in the case $\mathscr{L}_{\widetilde{E}_g} \otimes \mathscr{O}_{\widetilde{E}_g} \cong \mathscr{M}_{\widetilde{E}_g} \otimes \mathscr{O}_{\widetilde{E}_g} \cong \mathscr{O}_{\widetilde{E}_g}\left((2g-3)P'_g + P''_g\right)$, replace \eqref{eq:Divs4es} with
\begin{align*}
D_{2k+1}:= \,\,& (2g-3+a_r -2k)P'_g + (2k+1) P''_g \\ 
&\qquad \mbox{for } k=0,\dots,g+\Big\lfloor\frac{a_r-3}{2}\Big\rfloor,\\
D_{2k}:= \,\, & (2g-4+ a_r-2k)P'_g + 2k P''_g + Q+R \\ 
&\qquad \mbox{for } k=0,\dots,g-2+\Big\lfloor\frac{a_r}{2}\Big\rfloor,
\end{align*}
replace \eqref{eq:ordP''la} with 
\begin{equation*}
\mathrm{ord}_{P''_g}(\lambda_\ell) = \left\{
\begin{array}{ll}
2g-2-a_i-c_\ell & \quad\mbox{ if}\,\,\, a_i + c_\ell \equiv 1 \,\,\,\mathrm{mod} \,\,2,\\
2g-4-a_i-c_\ell & \quad\mbox{ if}\,\,\, a_i + c_\ell \equiv 0 \,\,\,\mathrm{mod} \,\,2,
\end{array}
\right. 
\end{equation*}
and similarly \eqref{eq:ordP''mu} with
\begin{equation*}
\mathrm{ord}_{P''_g}(\mu_m) = \left\{
\begin{array}{ll}
2g-2+a_i-b_m & \quad\mbox{ if}\,\,\, a_i + b_m \equiv 1 \,\,\,\mathrm{mod} \,\,2,\\
2g-4+a_i-b_m & \quad\mbox{ if}\,\,\, a_i + b_m \equiv 0 \,\,\,\mathrm{mod} \,\,2.
\end{array}
\right. 
\end{equation*}
After making these changes, the conclusion \eqref{eq:conclusion} becomes 
\[
c_\ell+b_m= 2g-2 \qquad\mbox{and} \qquad a_i + c_\ell \equiv  a_i + b_m \equiv 1 \,\,\,\mathrm{mod} \,\,2.
\]
Otherwise the argument proceeds  as in the previous case.
\end{proof}

\subsection{Proof of Theorem \ref{thm:prympetriateta}} 
For $\mathscr{L}_\eta \in \mathrm{V}^{\bm{a}}(\mathscr{C}_\eta, \epsilon_\eta, P_\eta)$ with $h^0(\widetilde{\mathscr{C}}_\eta, \mathscr{L}^i_\eta)=r+1-i$ for each $i$, consider the pointed Prym-Petri map $\overline{\mu}_\eta$ for $\mathscr{L}_\eta$ as in \eqref{eq:mubareta}. 
Let $\rho \in \mathrm{Ker}\left( \overline{\mu}_\eta\right)$. Assuming $\rho\neq 0$, one deduces an element $\overline{\rho}\neq 0$ as in \eqref{eq:nonzeroassumption}.
By Lemmata \ref{lem:ordPP} and \ref{lem:zeroconclusion}, one deduces $\overline{\rho}= 0$, a contradiction. 
\hfill$\square$


\section{Prym-Tyurin varieties for Prym-Brill-Noether curves}
\label{sec:expPT}

Here we prove Theorem \ref{thm:expPT}. The statement is inspired by a result from Ortega \cite{ortega2013brill}, showing that for a general curve $C$ of genus $2a+1$, 
$\mathrm{Jac}(C)$ is isomorphic as a principally polarized abelian variety to a Prym-Tyurin variety for the  Brill-Noether curve consisting of line bundles of degree $a+2$ with at least two independent global sections, of exponent computed  by the Catalan number $\frac{(2a)!}{a!(a+1)!}$.
The proof follows the argument from \cite{ortega2013brill} after appropriately translating the geometric objects from the Brill-Noether  to the Prym-Brill-Noether setting.

\begin{proof}[Proof of Theorem \ref{thm:expPT}]
Select $L_0\in \mathrm{V}^{\bm{a}}(C, \epsilon, P)$, and consider the embedding
\[
\mathrm{V}^{\bm{a}}(C, \epsilon, P) \hookrightarrow \mathrm{Pr} (C,\epsilon), \qquad L\mapsto L\otimes L_0^{-1}.
\]
From \cite[Remark 1.9]{fulton1981connectedness}, this induces a surjective map 
\begin{equation}
\label{eq:surjH1}
H_1(\mathrm{V}^{\bm{a}}(C, \epsilon, P),\mathbb{Z})\twoheadrightarrow H_1(\mathrm{Pr} (C,\epsilon), \mathbb{Z}).
\end{equation}
Since 
$H_1(\mathrm{Pr} (C,\epsilon), \mathbb{Z})\cong H_1(\widetilde{C}, \mathbb{Z})^-$, where $H_1(\widetilde{C}, \mathbb{Z})^-$ is the space of anti-invariant cycles, 
the rational extension of \eqref{eq:surjH1} induces a surjection
\[
\mathrm{Jac}(\mathrm{V}^{\bm{a}}(C, \epsilon, P)) \twoheadrightarrow \mathrm{Pr} (C,\epsilon).
\]
By duality, using the  respective principal polarizations  
\[
\mathrm{Pr} (C,\epsilon) \cong \mathrm{Pr} (C,\epsilon)^\vee \qquad \mbox{and}\qquad \mathrm{Jac}(\mathrm{V}^{\bm{a}}(C, \epsilon, P)) \cong \mathrm{Jac}(\mathrm{V}^{\bm{a}}(C, \epsilon, P))^\vee, 
\]
one has an embedding
\[
\mathrm{Pr} (C,\epsilon) \hookrightarrow \mathrm{Jac}(\mathrm{V}^{\bm{a}}(C, \epsilon, P)).
\]
 Welters's criterion \cite[12.2.2]{zbMATH02120946} yields that $(\mathrm{Pr} (C,\epsilon), \Xi)$ is isomorphic  to a Prym-Tyurin variety for the curve $\mathrm{V}^{\bm{a}}(C, \epsilon, P)$, and its exponent $e$ is determined by
$[\mathrm{V}^{\bm{a}}(C, \epsilon, P)] = \frac{e}{(g-2)!} [\Xi]^{g-2}$.
From Theorem \ref{thm:deglocus}, we deduce
\[
e= (g-2)! \, 2^{|\bm{a}|-\ell(\bm{a})}\, \prod_{i=0}^r \frac{1}{a_i!}\, \prod_{0\leq j<i\leq r}\frac{a_i-a_j}{a_i+a_j} .
\]
The assumption $\beta(g,\bm{a})=1$ implies $g-2=|\bm{a}|$, hence the statement. 
\end{proof}

\subsection*{Acknowledgements}
I would like to thank the referee for the interest in the manuscript and for the suggested edits.

\bibliographystyle{alphanumN}
\bibliography{Biblio}

\begin{thebibliography}{CHTiB}

\bibitem[ACT1]{act}
D. Anderson, L. Chen, and N. Tarasca.
\newblock {K}-classes of {B}rill-{N}oether loci and a determinantal formula.
\newblock {\em Internat. Math. Res. Notices, to appear, arXiv:1705.02992v3},
  2021.

\bibitem[ACT2]{anderson2020motivic}
D. Anderson, L. Chen, and N. Tarasca.
\newblock Motivic classes of degeneracy loci and pointed {B}rill-{N}oether
  varieties.
\newblock {\em Journal of the London Mathematical Society, to appear,
  arXiv:2010.05928}, 2021.

\bibitem[AF]{anderson2018chern}
D. Anderson and W. Fulton.
\newblock Chern class formulas for classical-type degeneracy loci.
\newblock {\em Compositio Mathematica}, 154(8):1746--1774, 2018.

\bibitem[BL]{zbMATH02120946}
C. {Birkenhake} and H. {Lange}.
\newblock {\em {Complex abelian varieties}}, volume 302.
\newblock Berlin: Springer, 2004.

\bibitem[CHTiB]{ciliberto1992endomorphisms}
C. Ciliberto, J. Harris, and M. Teixidor~i Bigas.
\newblock On the endomorphisms of {$\mathrm{Jac}(W^1_d (C))$} when {$\rho= 1$}
  and {$C$} has general moduli.
\newblock In {\em Classification of irregular varieties}, pages 41--67.
  Springer, 1992.

\bibitem[CP]{chan2021euler}
M. Chan and N. Pflueger.
\newblock Euler characteristics of {Brill-Noether} varieties.
\newblock {\em Transactions of the American Mathematical Society},
  374(3):1513--1533, 2021.

\bibitem[DCP]{de1995class}
C. De~Concini and P. Pragacz.
\newblock On the class of {B}rill-{N}oether loci for {P}rym varieties.
\newblock {\em Mathematische Annalen}, 302(1):687--697, 1995.

\bibitem[EH1]{MR723217}
D. Eisenbud and J. Harris.
\newblock A simpler proof of the {G}ieseker-{P}etri theorem on special
  divisors.
\newblock {\em Invent. Math.}, 74(2):269--280, 1983.

\bibitem[EH2]{MR846932}
D. Eisenbud and J. Harris.
\newblock Limit linear series: basic theory.
\newblock {\em Invent. Math.}, 85(2):337--371, 1986.

\bibitem[FL]{fulton1981connectedness}
W. Fulton and R. Lazarsfeld.
\newblock On the connectedness of degeneracy loci and special divisors.
\newblock {\em Acta Mathematica}, 146:271--283, 1981.

\bibitem[FP]{fulton2006schubert}
W. Fulton and P. Pragacz.
\newblock {\em Schubert varieties and degeneracy loci}.
\newblock Springer, 2006.

\bibitem[HH]{zbMATH00051939}
P.~N. {Hoffman} and J.~F. {Humphreys}.
\newblock {\em {Projective representations of the symmetric groups.
  {\(Q\)}-functions and shifted tableaux}}.
\newblock Oxford: Clarendon Press, 1992.

\bibitem[Mum1]{mumford1971theta}
D. Mumford.
\newblock Theta characteristics of an algebraic curve.
\newblock {\em Annales Scientifiques de l'{\'E}cole Normale Sup{\'e}rieure},
  4(2):181--192, 1971.

\bibitem[Mum2]{mumford1974prym}
D. Mumford.
\newblock Prym varieties {I}.
\newblock In {\em Contributions to analysis}, pages 325--350. Academic Press,
  Inc., 1974.

\bibitem[Ort]{ortega2013brill}
A. Ortega.
\newblock The {Brill--Noether} curve and {Prym--Tyurin} varieties.
\newblock {\em Mathematische Annalen}, 356(3):809--817, 2013.

\bibitem[PP]{pp}
A. Parusi\'nski and P. Pragacz.
\newblock Chern-{S}chwartz-{M}ac{P}herson classes and the {E}uler
  characteristic of degeneracy loci and special divisors.
\newblock {\em J. Amer. Math. Soc.}, 8(4):793--817, 1995.

\bibitem[Wel]{welters1985theorem}
G.~E. Welters.
\newblock A theorem of {G}ieseker-{P}etri type for {P}rym varieties.
\newblock In {\em Annales scientifiques de l'{\'E}cole Normale Sup{\'e}rieure},
  volume~18, pages 671--683, 1985.

\end{thebibliography}

\end{document}